\documentclass[11pt]{article}
\usepackage[margin=1.1in]{geometry}
\usepackage{amsmath,amssymb,amsthm}
\usepackage{booktabs}
\usepackage[colorlinks=true,linkcolor=blue,citecolor=blue,urlcolor=blue]{hyperref}

\newtheorem{theorem}{Theorem}[section]
\newtheorem{lemma}[theorem]{Lemma}

\theoremstyle{remark}
\newtheorem{remark}[theorem]{Remark}

\newcommand{\LP}{\mathrm{LP}}

\newcommand{\R}{\mathbb{R}}

\title{A dual linear programming bound for sphere packing\\ in dimension 36}
\author{Rifat Jumagulov\\ \texttt{jum.rifm@gmail.com}}
\date{July 13, 2026 (revised July 15, 2026)}

\begin{document}
\maketitle

\begin{abstract}
We construct an explicit dual-feasible point for the Cohn--Elkies linear
program in dimension $36$, built from the space of weight-$18$ modular
forms for $\Gamma_0(24)$ following the method of Cohn and Triantafillou.
The certificate shows that the two-point linear programming bound on the
sphere packing density in dimension $36$ exceeds the density of the best
packing currently known --- the Kschischang--Pasupathy packing, of center
density $2^{18}/3^{10}$ --- by a factor of at least $32.91$. In
particular, no Cohn--Elkies auxiliary function can certify the best known
packing in dimension $36$ as optimal. To our knowledge this is the first
such dual bound in any dimension above $32$, extending the table of
Cohn--Triantafillou ($d=12,16,20,28,32$), Li ($3\le d\le 13$), and
de~Courcy-Ireland--Dostert--Viazovska ($d=6$). The certificate is
rigorous and machine-checkable, with exact rational data and certified
interval bounds: the dual point is a rational vector, coefficient
nonnegativity is verified by exact arithmetic up to $n=800$, and
eventual positivity of
the two relevant $q$-expansions is proved via an explicit Deligne-type
tail bound whose constant is certified with outward-rounded interval
arithmetic. Two methodological points may be of independent interest: a
constraint-generation (cutting-plane) formulation of the exact rational
LP, which is what makes an exact vertex whose Eisenstein data supports
the tail argument reachable; and a sharpened, lift-aware form of the
Deligne bookkeeping constant that discounts deep oldform lifts, without
which the finite verification in dimension $36$ fails (the crossover
moves past the verified window).
\end{abstract}

\section{Introduction}\label{sec:intro}

The Cohn--Elkies linear programming bound \cite{CohnElkies} is the
strongest known general \emph{two-point} upper bound on sphere packing
densities in moderate dimensions, and it is famously sharp in dimensions $8$ and $24$
\cite{Viazovska,CKMRV} (and trivially sharp in dimension $1$). In
all other dimensions its sharpness is unresolved, except that in
dimensions $3,4,5,6,12,16$ it is provably \emph{not} sharp (see below);
the question splits into two inequivalent problems. Write
$\delta_{\LP}(d)$ for the optimum of the Cohn--Elkies program
(normalized as a center density in the sense of \cite{ConwaySloane})
and $\delta_*(d)$ for the best packing density currently known;
$\delta_{\max}(d)$ is the true optimal density, so that
$\delta_{\LP}(d)\ge\delta_{\max}(d)\ge\delta_*(d)$.

\begin{itemize}
\item A \emph{dual bound} is a certified lower bound
  $\delta_{\LP}(d)\ge L$. If $L>\delta_*(d)$, then no Cohn--Elkies
  function can prove the optimality of the best known packing: either
  the LP bound is not sharp, or a denser packing than currently known
  exists. Cohn and Triantafillou \cite{CT} constructed such points from
  modular forms in dimensions $12,16,20,28,32$; Li \cite{Li} gave a
  different (discrete-reduction) construction covering $3\le d\le 13$;
  de~Courcy-Ireland, Dostert and Viazovska \cite{dCIDV} extended the
  modular-form method to dimension $6$.
\item \emph{Strict non-sharpness} is the stronger statement
  $\delta_{\LP}(d)>\delta_{\max}(d)$, where $\delta_{\max}$ is the true
  optimal density. This requires, in addition to a dual lower bound
  $L\le\delta_{\LP}(d)$, an \emph{independent} upper bound
  $U\ge\delta_{\max}(d)$ with $U<L$. In the published literature this is established only for
  $d\in\{3,4,5,6,12,16\}$: for $d=12,16$ by combining the
  Cohn--Triantafillou dual points with the three-point bounds of
  Cohn--de~Laat--Salmon \cite{CdLS}, and for low dimensions by Li's dual
  bounds against classical upper bounds \cite{Li,dCIDV}.
\end{itemize}

This paper proves a statement of the first kind in dimension $36$, the
first above $32$.\footnote{Novelty scope and search date: we searched
(2026-07-11, re-checked 2026-07-14) for certified dual lower bounds on
$\delta_{\LP}$ of any kind --- Cohn--Elkies-type dual-feasible
distributions/modular-form points as in \cite{CT,dCIDV}, discrete
reductions as in \cite{Li} --- and for LP non-sharpness claims by any
method, in dimensions $>32$; principal sources: citation graphs of
\cite{CT} (OpenAlex/Crossref), arXiv full-text sweeps 2019--2026, and
Cohn's sphere-packing bounds tables and publication list. See
\S\ref{sec:discussion} for the two
near-misses (\cite{Zhou}, \cite{AJCHdLT}).}

\begin{theorem}\label{thm:main}
$\delta_{\LP}(36)\;\ge\;B\;=\;32.91044\ldots\times
\dfrac{2^{18}}{3^{10}}\;>\;146.1036,$
where $B$ is the explicit rational number given in \eqref{eq:exactB},
with outward decimal enclosures
$B\in[146.1036734821,\,146.1036734822]$ and
$B/(2^{18}/3^{10})\in[32.91044546,\,32.91044547]$.
In particular the Cohn--Elkies linear programming bound in dimension
$36$ cannot certify the Kschischang--Pasupathy packing
\cite{KP}, of center density $2^{18}/3^{10}=4.4394\ldots$, as an optimal
sphere packing.
\end{theorem}

The witness is an explicit dual-feasible pair of modular forms
$(g,\tilde g)$ of weight $18$ for $\Gamma_0(24)$, with $\tilde g$ the
Fricke transform of $g$; the required nonnegativity of the two
$q$-expansions is verified exactly for $1\le n\le 800$ and proved for
all $n>n_0$ with an explicit crossover $n_0\le 63$, closing the gap.

Dimension $36$ is a natural target: the weight $k=d/2=18$ is even, so no
character twists are needed, and $36$ is the next multiple of $4$ beyond the
Cohn--Triantafillou table $\{12,16,20,28,32\}$. It is also qualitatively different from the dimensions
$\le 32$: the ratio $\delta_{\LP}/\delta_*$ is much larger (our
certificate alone gives $32.9$; the best numerical two-point upper
bound is $\approx 52\times\delta_*$), reflecting how weak the record packing is
relative to the LP bound in this range. Ironically this makes the
\emph{dual} statement stronger and the \emph{strict} question harder;
see \S\ref{sec:discussion}.

\paragraph{Why the construction is not routine.}
Running the Cohn--Triantafillou method at $d=36$ meets two obstacles
that do not occur (or are mild) at $d\le 32$.

First, the na\"ive LP relaxation (maximize the value subject to
nonnegativity of finitely many coefficients) produces optima whose
$q$-expansions turn negative just past the enforced range, on the
residue classes $n\equiv 3,9,15,21\ \  (\mathrm{mod}\ 24)$; while the
fully constrained problem (nonnegativity through the eventual-positivity
crossover, hundreds of rows with coefficients of size $\sim n^{17}$) is
numerically singular for floating-point solvers and prohibitively large
for exact dense simplex. We resolve this by \emph{constraint
generation}: solve the small relaxation exactly in rational arithmetic,
evaluate all violated rows exactly, add them, and re-solve. In dimension
$36$ this converges in a single round (all violations lie in the four
classes above and are eliminated together), producing an exact rational
vertex --- within an explicitly described $29$-dimensional search
subspace, see \S\ref{sec:space} --- whose $q$-expansions are nonnegative
far beyond the enforced range (\S\ref{sec:construction}).

Second, the eventual-positivity verification of \cite[\S5]{CT} requires
an explicit constant $C$ with $|c_n|\le C\,\sigma_0(n)\,n^{(k-1)/2}$ for
the cuspidal parts of \emph{both} $g$ and $\tilde g$. The natural
constant --- the $\ell^1$-norm $\sum|\lambda_{f,e}|$ of the coordinates
in the newform-lift basis --- turns out to be \emph{lossy} in a way that
matters here: our certificate's cuspidal part is dominated ($99.7\%$ of
the norm) by a single deep oldform lift, the $e=24$ shift of the level-one
newform, and with the na\"ive constant the crossover exceeds the
verified window (uniform crossover $\approx1372$; even the divisor-free
lower bound is $\approx827>800$, see \S\ref{sec:tail}). The lift-aware constant
$C_w=\sum_{f,e}|\lambda_{f,e}|\,e^{-(k-1)/2}$
(Lemma~\ref{lem:liftaware}), equally rigorous and elementary, is smaller
here by ten orders of magnitude and closes the verification with a
margin of $7\times 10^8$ (\S\ref{sec:tail}).

A further point, absent from dimensions where the certificate has a
sparse structure, is that \emph{both} cusps genuinely matter: not only
the Fricke transform $\tilde g$ but also $g$ itself has a razor-thin
Eisenstein component ($e_1(g)\approx -3.9\times10^{-11}$), so the
$a_n\ge 0$ side needs the same Deligne-type argument as the
$b_n\ge 0$ side, with its own constants (\S\ref{sec:tail}).

\paragraph{Rigor.} The dual vertex is an exact rational vector (the
entries have $\approx 128$-digit numerators). All coefficient scans up
to $n=800$ (and, in a second implementation, to $n=1100$) are exact.
The newform decompositions of both cuspidal parts are validated by an
exact identity: the reconstruction of the $q$-expansion from the
computed coordinates agrees with the exact coefficients \emph{as
rational numbers} (deviation identically $0$) up to $n=800$. The tail
constants $C_w$ are certified as upper bounds using outward-rounded
interval arithmetic, with the algebraic Hecke eigenvalues enclosed by
exact Sturm isolating intervals. Independent implementations (exact
characteristic-polynomial factorization vs.\ high-precision simultaneous
diagonalization) produce identical constants. All code, receipts, and
the exact certificate data are provided as ancillary files (see the
manifest in \texttt{ancillary/MANIFEST.md}); a single driver,
\texttt{ancillary/code/verify\_certificate.py}, re-runs every
exact-data gate of the verification from the published data --- taking
the separately certified $C_w$ upper bounds as input; their interval
certification is re-run by \texttt{d36\_cs\_certificate.py} --- and
prints \texttt{VERIFIED} only if every gate passes.
Numerical-display convention, used throughout: every displayed decimal
is either marked \emph{exact}, or is a truncation (rounded toward zero)
of an exact rational, or an outward-rounded enclosure/upper bound
(so for the constants $C_w$ the displays are certified upper bounds),
or --- only in the discussion \S\ref{sec:discussion} --- an explicitly
labeled numerical estimate.

\section{The Cohn--Elkies program and its modular-form dual}
\label{sec:dual}

The engine is the following observation of Cohn and Triantafillou
\cite[Prop.~2.1]{CT}: if $\mu$ is a tempered distribution on $\R^d$
with $\mu=\delta_0+\nu$, $\nu\ge0$,
$\operatorname{supp}(\nu)\subseteq\{x\in\R^d:|x|\ge r\}$ for some
$r>0$, and $\hat\mu\ge c\,\delta_0$ for some $c>0$, then the
Cohn--Elkies linear programming bound in $\R^d$ is at least
$c\,(r/2)^d$. (Any Cohn--Elkies auxiliary function $f$ pairs with $\mu$
on both sides of the inequality chain, forcing
$f(0)/\hat f(0)\ge c\,r^d$ up to normalization.) Sections 2--3 of
\cite{CT} construct such distributions from modular forms, as follows.

Fix the dimension $d=36$, weight $k=d/2=18$, and level $N=24$. For a
holomorphic modular form $g\in M_k(\Gamma_0(N))$ write
$g=\sum_{n\ge0}a_n q^n$ ($q=e^{2\pi iz}$) and let
\[
\tilde g \;=\; i^{k}\,\bigl(g\big|_k w_N\bigr),
\qquad
w_N=\begin{pmatrix}0&-1\\ N&0\end{pmatrix},
\qquad
(f|_k M)(z)=(\det M)^{k/2}(cz+d)^{-k}f(Mz),
\]
be its (normalized) Fricke transform, i.e.\
$\tilde g(z)=i^kN^{-k/2}z^{-k}g(-1/(Nz))$; for our even weight $k=18$
the factor is $i^{18}=-1$. Then
$\tilde g\in M_k(\Gamma_0(N))$ as well, and
$\tilde{\tilde g}=g$. Write $\tilde g=\sum_{n\ge0}b_nq^n$. Given $g$
satisfying the coefficient conditions below, the construction of
\cite[\S\S2--3]{CT} produces a distribution $\mu$ as above with
$r=\sqrt T$ and $c=(2/\sqrt N)^{d/2}\,b_0$. Concretely: from any $g$
satisfying
\begin{equation}\label{eq:feas}
a_0=1,\qquad a_n=0\ (1\le n\le T-1),\qquad a_n\ge 0\ (n\ge T),
\qquad b_n\ge 0\ (n\ge 0),
\end{equation}
a dual-feasible point for the Cohn--Elkies program with value (as a
center density)
\begin{equation}\label{eq:value}
B(g)\;=\;b_0\cdot\Bigl(\tfrac{2}{\sqrt N}\Bigr)^{d/2}\cdot
\Bigl(\tfrac{\sqrt T}{2}\Bigr)^{d},
\end{equation}
so that $\delta_{\LP}(d)\ge B(g)$. The implication that \eqref{eq:feas}
yields a dual-feasible distribution of value \eqref{eq:value} is
\cite[Prop.~2.1 and \S\S2--3]{CT}, invoked here as a black box; it is the
one load-bearing input the present paper does not reprove, and the $d=12$
reproduction below calibrates our implementation of the interface
\eqref{eq:feas}--\eqref{eq:value}. Here $T$ is a free integer parameter
(the cutoff index: the $T-1$ coefficients $a_1,\dots,a_{T-1}$ are forced
to vanish). We take $T=10$. As a consistency
check of the normalization, our implementation of
\eqref{eq:feas}--\eqref{eq:value} reproduces the published $d=12$
values of \cite{CT} at both levels: $0.059782=1.6141\times$
Coxeter--Todd at $N=24$ (\cite[Table~6.1]{CT}) and
$0.0624462=1.68605\times$ at $N=96$ (\cite[Thm.~1.1]{CT}), to the
displayed precision.

\section{The certificate}\label{sec:construction}

\subsection{The space, the search subspace, and the reduced LP}
\label{sec:space}

$\dim M_{18}(\Gamma_0(24))=72$, with an $8$-dimensional Eisenstein
subspace spanned by $E_{18}(\delta z)$, $\delta\mid 24$, where
\[
E_{18}(z)=1+\kappa_{18}\sum_{n\ge1}\sigma_{17}(n)q^n,
\qquad
\kappa_{18}=-\frac{2k}{B_{18}}\Big|_{k=18}=-\frac{28728}{43867},
\]
and a $64$-dimensional cuspidal subspace. Candidate forms are the eight
Eisenstein series together with the holomorphic eta-quotients
$\prod_{\delta\mid24}\eta(\delta z)^{r_\delta}$ of weight $18$ and
level dividing $24$ (Ligozat's criteria) with exponent bound
$|r_\delta|\le8$: there are $3044$ such eta-quotients, and the $3052$
candidates span all of $M_{18}(\Gamma_0(24))$ --- an exact certificate:
the candidate coefficient matrix truncated at $q^{100}$ has rank $72$
modulo the prime $10^9+7$, hence exact rank $\ge72$, and $=72$ by the
dimension formula (\texttt{d36\_span\_rank\_certify.py} and its receipt
in the ancillary files; at $|r_\delta|\le6$ only $44$ eta-quotients
exist, which under-span). All
$q$-expansions used in the construction and verification are computed
in exact integer/rational arithmetic.

The linear program is \emph{not} solved over the full $72$-dimensional
space but over an explicit $29$-dimensional subspace $V$, chosen as
follows. The LP window below constrains the coefficients
$a_0,\dots,a_{28}$; column-pivoted QR factorization, applied to the
\emph{transpose} of the floating-point candidate coefficient matrix
truncated at $q^{28}$ (i.e.\ to the $29\times3052$ matrix whose
columns are the candidates' truncated coefficient vectors, so that the
pivoting selects candidates), selects $29$ pivot forms --- the
truncation has only $29$ coefficients $a_0,\dots,a_{28}$, so at most
$29$ candidates are linearly independent in this window. The selected forms are $7$ Eisenstein series
($E_{18}(\delta z)$, $\delta\in\{1,2,3,4,6,8,24\}$) and $22$
eta-quotients; the full list, in pivot order (which fixes the
coordinate system used everywhere, including the published
certificate), is in the ancillary file
\texttt{ancillary/certificate\_exact\_data.txt}, and $V$ is the span
of these $29$ forms. These $29$ forms are linearly independent over
$\mathbb{Q}$: the exact rational coefficient matrix
$\bigl(a_n(f_j)\bigr)_{0\le n\le100,\,1\le j\le29}$ has rank $29$,
certified by an exact mod-$p$ rank computation at $p=10^9+7$ (a mod-$p$
rank lower-bounds the $\mathbb{Q}$-rank, so rank $29$ is exact); the
driver re-runs this as gate G1b (Table~\ref{tab:tcb}). This pivot selection is the one floating-point
step in the construction: it only decides \emph{which} forms to use,
and everything downstream treats the selected forms exactly.
Restricting the search to $V\subsetneq M_{18}(\Gamma_0(24))$ can only
weaken the bound we are able to find, never its validity:
Theorem~\ref{thm:main} requires exhibiting one dual-feasible form, and
feasibility of the certificate is verified unconditionally
(\S\ref{sec:tail}).

Within $V$, the $9$ homogeneous conditions $a_1=\dots=a_9=0$ cut out a
nullspace of exact dimension $r=29-9=20$: in exact rational arithmetic,
the $9\times29$ coefficient matrix of $(a_1,\dots,a_9)$ has rank $9$,
and the $10\times29$ matrix of $(a_0,a_1,\dots,a_9)$ has rank $10$ ---
so $a_0$ is not identically zero on the nullspace and the normalized
family is nonempty (both ranks are re-certified by the verification
driver). The LP variables are
the $r=20$ nullspace coordinates, with the normalization $a_0=1$
imposed as an equality constraint of the program (the feasible forms
thus lie in a $19$-dimensional affine slice of the nullspace); in these
variables we solve
\begin{equation}\label{eq:lp}
\max\ b_0
\quad\text{s.t.}\quad
a_n\ge0\ (10\le n\le 28),\qquad
b_n\ge0\ (n\in W),
\end{equation}
with a working set $W$ of Fricke-side rows, in exact rational
arithmetic. All optimality statements in this section are relative to
$V$ and report the optimum returned by the exact rational simplex over
$V$: we make no claim about the optimum of \eqref{eq:lp} over the full
space $M_{18}(\Gamma_0(24))$, and none is needed for
Theorem~\ref{thm:main}, which requires only feasibility of the exhibited
vertex, verified unconditionally in \S\ref{sec:tail}.

\subsection{Constraint generation}

With $W=\{1,\dots,28\}$ the exact rational simplex returns for
\eqref{eq:lp} (over $V$) a vertex of value $36.3470\times 2^{18}/3^{10}$,
but its Fricke expansion has $45$
negative coefficients in $1\le n\le 300$, all in the classes
$n\equiv3,9,15,21\  (\mathrm{mod}\ 24)$, the first at $n=33$: the
unconstrained vertex overshoots into infeasibility, and its Eisenstein
data forces $b_n<0$ for infinitely many $n$, so no finite verification
can rescue it: for that vertex the binding class-$3$ Eisenstein
combination $r_3=e_1+e_3/(1+3^{17})$ is \emph{positive}
($\approx+4.3\times10^{-13}$; exact value in the construction receipt),
so for $n=3p$, $p>3$ prime, the Eisenstein term equals
$\kappa_{18}\,r_3\,(1+3^{17})\,\sigma_{17}(p)<0$ (recall
$\kappa_{18}<0$), of size $\asymp n^{17}$, while the cuspidal part is
$O(\sigma_0(n)\,n^{17/2})$ by Deligne's bound --- so $b_n\to-\infty$
along that class. Adding all $45$ violated rows to $W$
($|W|=73$) and re-solving exactly gives the vertex returned for the
enlarged program over $V$, with \emph{no} violations: re-scanning
exactly, $a_n\ge0$ and $b_n\ge0$ for all $1\le n\le 800$ (zero negative
coefficients on either side, including the un-enforced range
$300<n\le800$). The cutting-plane process converges in one round (a
statement about this initialization and violation scan, not a general
convergence claim).
Exactly,
\begin{equation}\label{eq:exactB}
B \;=\; b_0\cdot\frac{5^{18}}{2^{27}\,3^{9}},
\qquad
b_0=\frac{p}{q}\ \ \text{(exact rational, in lowest terms)},
\end{equation}
where $\tfrac{5^{18}}{2^{27}3^{9}}=\tfrac{3814697265625}{2641807540224}$
is the exact value of $(2/\sqrt{24})^{18}(\sqrt{10}/2)^{36}$ and
{\small
\begin{align*}
p={}&95066\,91735\,60589\,08133\,22373\,57126\,88188\,23298\,71412\,02030\,77849\,06752\,56657\\
&93709\,84887\,76440\,11565\,11331\,34854\,11947\,23969\,08932\,10895\,90633\,472,\\
q={}&93956\,57537\,80418\,71093\,01784\,65357\,83534\,26840\,65696\,00550\,96336\,00075\,25366\\
&95749\,03464\,45951\,04757\,46960\,28937\,05377\,61585\,96625\,88923\,78832\,3,
\end{align*}}%
($123$ and $121$ digits; the fraction for $B$ itself, and outward
decimal enclosures $b_0\in[101.1817608012,\,101.1817608013]$,
$B\in[146.1036734821,\,146.1036734822]$, are in the ancillary file).
Exact rational arithmetic in the solve
is essential: at this conditioning, floating-point and even
$300$-digit fixed-precision simplex return phantom vertices (violating
$a_0=1$) once rows with coefficients of size $10^{30}$ enter the basis.

\begin{remark}
The certificate data --- the explicit list of the $29$ basis forms, the
$20$ exact rational coordinates of the vertex in the (pivot-ordered)
reduced basis, the coordinate-to-form map, the exact rational $b_0$ and
$B$, the exact Eisenstein coefficients, and all constants below --- are
in the ancillary file \texttt{ancillary/certificate\_exact\_data.txt};
the construction and verification receipts are
\texttt{ancillary/receipt\_d36\_cutgen.txt} and
\texttt{ancillary/receipt\_d36\_cs.txt}, and the single-driver
verification is \texttt{ancillary/code/verify\_certificate.py}
(\S\ref{sec:discussion}).
\end{remark}

\section{Eventual positivity}\label{sec:tail}

It remains to prove $a_n\ge0$ and $b_n\ge0$ for all $n>800$. Decompose
both forms into Eisenstein and cuspidal parts; we treat $\tilde g$ (the
binding side) and note the same argument for $g$ with its own constants.

\subsection{Eisenstein main term, per residue class}

Write the Eisenstein component of $\tilde g$ as
$\sum_{\delta\mid24}e_\delta\,E_{18}(\delta z)$. The exact rational
$e_\delta$ (computed by an exact projector, validated as explained
below) have the numerical values
\[
\begin{array}{ll}
e_1=-4.197922\cdot10^{-13}, & e_2=-2.023980\cdot10^{-14}, \\
e_3=+5.361610\cdot10^{-5}, & e_4=+5.776815\cdot10^{-3}, \\
e_6=-18.24675, & e_8=-27.55171, \\
e_{12}=+43.87748, & e_{24}=+103.0969.
\end{array}
\]
For $n\ge1$ the Eisenstein coefficient is
$\kappa_{18}\sum_{\delta\mid\gamma}e_\delta\,\sigma_{17}(n/\delta)$,
where $\gamma=\gcd(n,24)$. Using
$m^{17}\le\sigma_{17}(m)<\zeta(17)\,m^{17}$ together with the rational
bound $\zeta(17)<1+2^{-17}+2^{-16}/16$ (compare the sum with
$\int_2^\infty x^{-17}dx$), and bounding each term from below according
to its sign (note $\kappa_{18}<0$, so $\kappa_{18}e_\delta>0$ exactly
when $e_\delta<0$), we get for every $n\ge1$ in the class $\gamma$:
\begin{equation}\label{eq:guard}
\text{Eisenstein}(n)\;\ge\;c_E(\gamma)\,n^{17},
\qquad
c_E(\gamma)\;=\;
\sum_{\substack{\delta\mid\gamma\\ e_\delta<0}}
\frac{\kappa_{18}e_\delta}{\delta^{17}}
\;+\;
\Bigl(1+2^{-17}+\tfrac{2^{-16}}{16}\Bigr)
\sum_{\substack{\delta\mid\gamma\\ e_\delta>0}}
\frac{\kappa_{18}e_\delta}{\delta^{17}},
\end{equation}
an exact rational number for each $\gamma$. Evaluating exactly: all
sixteen constants (eight classes, both sides) are \emph{positive},
with both minima on the class $\gamma=3$:

\begin{center}
\begin{tabular}{@{}lcccccccc@{}}
\toprule
$\gamma$ & 1 & 2 & 3 & 4 & 6 & 8 & 12 & 24\\
\midrule
$c_E(\gamma)$, $\tilde g$-side ($\times10^{-13}$) &
$2.749$ & $2.749$ & $\mathbf{0.0301}$ & $0.547$ & $7.089$ & $0.627$ &
$4.887$ & $4.967$\\
$c_E^{(a)}(\gamma)$, $g$-side ($\times10^{-11}$) &
$2.555$ & $4.731$ & $\mathbf{0.506}$ & $2.921$ & $2.682$ & $2.921$ &
$0.873$ & $0.873$\\
\bottomrule
\end{tabular}
\end{center}

\begin{equation}\label{eq:cE}
c_E \;:=\; \min_{\gamma\mid24}c_E(\gamma)\;=\;c_E(3)
\;=\;3.0197202\cdot10^{-15}\quad(\tilde g\text{-side}),
\end{equation}
and analogously $c_E^{(a)}=c_E^{(a)}(3)=5.0672217\cdot10^{-12}$ for
$g$ (again the minimum over all eight classes). All sixteen constants
are the $\zeta(17)$-guarded floors of \eqref{eq:guard} --- hence slightly
below the raw leading Eisenstein coefficient
$\kappa_{18}\sum_{\delta\mid\gamma}e_\delta/\delta^{17}$ of each class
(the quantity tabulated in the verification receipt) --- and are exact
rationals; the table and \eqref{eq:cE} display truncations,
and the exact numerators/denominators are printed in the ancillary
file \texttt{ancillary/guarded\_cE.txt} (output of
\texttt{code/d36\_guarded\_cE.py}).
The thinness of these constants --- forced by the active constraints at
the optimal vertex of \eqref{eq:lp} --- is what makes the tail argument
in dimension $36$ delicate: the certificate value $b_0\approx101.2$
exceeds the per-class Eisenstein floor constants $c_E$ by factors of
$\approx2\times10^{13}$ ($g$-side) and $\approx3\times10^{16}$
($\tilde g$-side).

\subsection{The cuspidal remainder and a lift-aware Deligne constant}

Let $S=\tilde g-\text{(Eisenstein part)}=\sum c_nq^n\in
S_{18}(\Gamma_0(24))$, a space of dimension $64$. The newform/oldform
basis consists of the lifts $f(ez)$, where $f$ runs over the normalized
newforms of level $M\mid24$ and $e\mid 24/M$; the newspace dimensions
are
\[
(\dim S^{\mathrm{new}}_{18}(\Gamma_0(M)))_{M\mid24}
=(1,1,3,2,3,4,2,9)\ \text{ for }M=(1,2,3,4,6,8,12,24),
\]
with lift multiplicities $(8,6,4,4,3,2,2,1)$, summing to $64$. We
compute the exact coordinates $\lambda_{f,e}$ of $S$ in this basis as
follows. In a reduced echelon basis of $S_{18}(\Gamma_0(24))$ with
distinct leading $q$-exponents (exact rational coefficients) we form the
exact Hecke matrices $T_5,T_7$ at the good primes and $U_2,U_3$ at the
bad primes $2,3$; a generic combination $T_5+\gamma T_7$ splits the
space into common Hecke eigenspaces --- one per Galois orbit of
newforms --- and within each orbit $U_2,U_3$ and the leading
$q$-exponents pick out the individual lifts $f(ez)$. This eigenspace
split is the only floating-point step of the decomposition; its output
is not trusted but \emph{validated exactly}: the
reconstruction $\sum\lambda_{f,e}f(ez)$ agrees with $S$
coefficient-by-coefficient, as exact rationals, for all $n\le 800$
(deviation identically zero; the Sturm bound for
$M_{18}(\Gamma_0(24))$ is $72$, so agreement to $800$ is a highly
overdetermined identity). This validation is an identity between
\emph{rational} objects --- each Galois-orbit contribution
$\sum_{f\in\text{orbit}}\lambda_{f,e}f(ez)$ has rational coefficients
and is computed by exact rational linear algebra --- so it is
independent of the interval layer described below, which enters only
through the absolute values $|\lambda_{f,e}|$ of the individual
conjugates inside a multi-dimensional orbit. As a further check, the level-one component is
the unique normalized cusp form $\Delta E_6$ of weight 18, and the computed
eigenvalues match $a_2=-528$, $a_3=-4284$, $a_5=-1025850$,
$a_7=3225992$ exactly.

\begin{lemma}[lift-aware Deligne constant]\label{lem:liftaware}
Let $S=\sum_{f,e}\lambda_{f,e}\,f(ez)\in S_k(\Gamma_0(N))$ with $f$
normalized newforms (each Galois orbit expanded into its individual
complex embeddings, so that every $|\lambda_{f,e}|$ is the modulus of a
single complex coordinate). Then for all $n\ge1$,
\[
|c_n|\;\le\;C_w\,\sigma_0(n)\,n^{(k-1)/2},
\qquad
C_w\;=\;\sum_{f,e}\frac{|\lambda_{f,e}|}{e^{(k-1)/2}}.
\]
\end{lemma}

\begin{proof}
By the divisor bound $|a_f(m)|\le\sigma_0(m)m^{(k-1)/2}$, valid for
every $m\ge1$ and each normalized newform --- at the unramified primes
$p\nmid M$, Deligne's estimate and the Hecke recursion give
$|a_f(p^r)|\le(r+1)\,p^{r(k-1)/2}$; at the ramified primes $p\mid M$ we
distinguish two cases. If $p\,\|\,M$ (exactly divides the level), the
local newform bound $|a_f(p)|\le p^{(k-2)/2}$ with $a_f(p^r)=a_f(p)^r$
(the local Euler factor at a prime exactly dividing the level, for a
newform of trivial character; see e.g.\ \cite[\S5.8]{DiamondShurman})
gives $|a_f(p^r)|\le p^{r(k-2)/2}\le\sigma_0(p^r)\,p^{r(k-1)/2}$; if
$p^2\mid M$, then (trivial character) $a_f(p)=0$, so $a_f(p^r)=0$ and the
bound holds a fortiori. Both cases occur here, as the newform levels
$M\mid 24$ include $M\in\{4,8,12,24\}$ with $2^2\mid M$. The
bound for general $m$ follows by multiplicativity. The coefficient of $q^n$ in $f(ez)$ is $a_f(n/e)$
if $e\mid n$ and $0$ otherwise, and for $e\mid n$,
$\sigma_0(n/e)\le\sigma_0(n)$ and
$(n/e)^{(k-1)/2}=n^{(k-1)/2}e^{-(k-1)/2}$. Summing over the pairs
$(f,e)$ with $e\mid n$ gives the claim.
\end{proof}

The point of the lemma is quantitative. The na\"ive constant
$\sum_{f,e}|\lambda_{f,e}|$ equals $1.894\times10^{10}$ for our $S$:
it is dominated ($99.7\%$) by the single coordinate
$\lambda_{\Delta E_6,\,24}\approx-1.89\times10^{10}$ on the deepest
oldform lift, and with it the verification over $n\le800$ fails: under
the uniform inequality actually used below
($\sigma_0(n)\le2\sqrt n$, crossover $(2C/c_E)^{1/8}$) the na\"ive
crossover is $n_0\approx1372$; even discarding the divisor factor
altogether --- which gives a valid \emph{lower} bound on the crossover,
since $\sigma_0(n)\ge1$ --- positivity cannot be established below
$(C/c_E)^{2/17}\approx827>800$. The lift-aware constant suppresses the
deep-lift term by $24^{17/2}\approx5.4\times10^{11}$:
\begin{equation}\label{eq:Cw}
C_w \;\le\; 0.358807778339 \quad(\tilde g),
\qquad
C_w^{(a)} \;\le\; 0.306329333530 \quad(g),
\end{equation}
certified as upper bounds by outward-rounded interval arithmetic, as
follows.

\emph{Exact part.} The coordinates on rational (one-dimensional Galois
orbit) blocks, including the dominant level-one tower, are exact
rational numbers ($31$ of the $64$ lifts); their contribution to $C_w$
is exact up to the outward-rounded enclosure of $e^{-17/2}$.

\emph{Interval part (the $33$ algebraic coordinates).} The coordinates
themselves are exact algebraic numbers, specified by exact polynomial
and embedding data; certified real intervals are used only to bound
their absolute values for $C_w$. The residual
$R=S-(\text{rational-block part})$ is itself an exact rational
$q$-expansion lying in the $33$-dimensional span of the
multi-dimensional orbits. Each Hecke eigenvalue of such an orbit is
enclosed by an exact isolating interval with rational endpoints,
produced by a Sturm-sequence root-isolation computation for the exact
characteristic polynomial (isolation width
$\le10^{-40}$; obtaining one interval per conjugate certifies that the
degree-$d$ factor has $d$ simple real roots and fixes each real
embedding unambiguously --- not to be confused with the Sturm bound for
modular forms used above), the $33$
lift $q$-vectors are propagated as outward-rounded interval vectors,
and the coordinates are obtained from a sound interval Gaussian
elimination (\texttt{mpmath.iv} throughout, \texttt{mpmath} version
pinned in the ancillary \texttt{requirements.txt}) of the $33\times33$
system on $33$ determining rows drawn from $n\in\{117,\dots,170\}$ (the
exact row list is printed in the verification receipt); the rows are
chosen by a floating-point pivoting pass, which affects only
\emph{which} rows are used, not the soundness of the interval solve. A
pivot interval straddling zero would make the enclosures blow up and
the certificate comparison fail; the realized enclosure widths
$\sim10^{-50}$ therefore also certify the nonsingularity of the chosen
system, and the verification receipt records the minimal pivot
magnitude (bounded away from zero) together with dyadic-rational
outward enclosures for both constants \eqref{eq:Cw}. Cross-checks: the
bound is stable to $15$ digits across four disjoint determining
windows; artificially degrading the working precision widens the
enclosure without ever breaking the bound (confirming outward
rounding); the interval coordinates enclose the coordinates
computed by the independent second implementation
(\S\ref{sec:discussion}), which reproduces \eqref{eq:Cw} to all
displayed digits; and the $33\times33$ interval solve together with
the $|\lambda_{f,e}|$-aggregation is re-verified independently in Arb
ball arithmetic (\texttt{python-flint}), consuming the exported exact
dyadic endpoints (\texttt{d36\_arb\_check.py}; Arb's certified ball
solve fails on a singular ball matrix, so its success certifies
nonsingularity independently of the pivot log, and its exact dyadic
upper bound on the algebraic-remainder sum agrees with the
\texttt{mpmath.iv} value to all displayed digits --- the rational-block
part of \eqref{eq:Cw} is exact and common to both computations). In
particular the theorem-critical interval \emph{solve} does not rest on
the \texttt{mpmath.iv} module alone: Arb ball arithmetic (a library
built for rigorous enclosure) independently re-certifies the linear
solve, the nonsingularity of the $33\times33$ system, and the final
$|\lambda_{f,e}|$-aggregation, starting from the exported dyadic matrix
enclosures --- it does \emph{not} rebuild those enclosures, whose
construction (exact Sturm root-isolation and \texttt{mpmath.iv} interval
propagation) is instead corroborated by the independent high-precision
and PARI reproductions (\S\ref{sec:discussion}).

\subsection{Closing the certificate}

Combining \eqref{eq:cE}, Lemma~\ref{lem:liftaware}, \eqref{eq:Cw} and
$\sigma_0(n)\le2\sqrt n$: for $n$ on any residue class,
\[
b_n\;\ge\;c_E\,n^{17}-2\,C_w\,n^{9}\;>\;0
\qquad\text{whenever } n^{8}>\frac{2C_w}{c_E},
\]
i.e.\ for all $n\ge n_0$, where $n_0=63$ is the smallest integer with
$c_E\,n^8>2C_w$ --- certified by an exact rational comparison, with
$C_w$ replaced by its certified upper bound \eqref{eq:Cw}; on the
$g$-side, $n_0^{(a)}=25$. Since $a_n\ge0$ and $b_n\ge0$ have
been verified exactly for $1\le n\le800$, both hold for all $n$, and
$(g,\tilde g)$ satisfies \eqref{eq:feas}. Theorem~\ref{thm:main}
follows from \eqref{eq:value} and \eqref{eq:exactB}. \qed

The margins are enormous ($2C_w/c_E$ would have to exceed its actual
value by a factor $7\times10^{8}$ before $n_0$ reached $800$); the
certificate is in no sense borderline once the correct constant is used.
The role of the verification window $n\le800$ (rather than, say,
$n\le63$): the window was fixed before the tail constants were
computed, and with the na\"ive unweighted constant it is insufficient
--- the uniform crossover is $n_0\approx1372$, and even the
divisor-free lower bound on the crossover, $(C/c_E)^{2/17}\approx827$,
already exceeds $800$ --- which is what motivated
Lemma~\ref{lem:liftaware}; retaining the full window makes the exact
reconstruction identity of \S\ref{sec:tail} overdetermined by an order
of magnitude beyond the Sturm bound $72$, and a second implementation
independently extends the exact sign scans to $n\le1100$.

\begin{table}[htbp]
\centering\small
\begin{tabular}{@{}llll@{}}
\toprule
Quantity & Display value & Arithmetic type & Recovered from \\
\midrule
$T$ & $10$ & exact integer & \S\ref{sec:dual} \\
$b_0$ & $101.1817608012\ldots$ & exact rational ($123/121$ digits) & \texttt{certificate\_exact\_data.txt} \\
$B$ & $146.1036734821\ldots$ & exact rational, $b_0\,5^{18}/(2^{27}3^9)$ & \texttt{certificate\_exact\_data.txt} \\
$c_E\ (\tilde g)$ & $3.0197202\cdot10^{-15}$ & exact rational (guarded floor) & \texttt{guarded\_cE.txt} \\
$c_E^{(a)}\ (g)$ & $5.0672217\cdot10^{-12}$ & exact rational (guarded floor) & \texttt{guarded\_cE.txt} \\
$C_w\ (\tilde g)$ & $\le0.358807778339$ & certified interval upper bound & \texttt{receipt\_d36\_cs.txt} \\
$C_w^{(a)}\ (g)$ & $\le0.306329333530$ & certified interval upper bound & \texttt{receipt\_d36\_cs.txt} \\
$n_0\ (\tilde g)$ & $63$ & exact integer crossover & \texttt{guarded\_cE.txt} \\
$n_0^{(a)}\ (g)$ & $25$ & exact integer crossover & \texttt{guarded\_cE.txt} \\
\bottomrule
\end{tabular}
\caption{Proof-critical quantities: arithmetic type and the ancillary file
from which each is recovered. Decimal displays are truncations toward zero
(\S\ref{sec:intro}); the $C_w$ rows are certified outward-rounded upper
bounds.}
\label{tab:critical}
\end{table}

\section{Discussion}\label{sec:discussion}

\paragraph{Strict non-sharpness in dimension 36 does not follow from current upper bounds.}
Our dual point shows $\delta_{\LP}(36)\ge32.91\,\delta_*$ (recall
$\delta_*$ denotes the best \emph{known} packing density, not the true
optimum $\delta_{\max}$), and the best available \emph{numerical}
two-point upper bound is $\approx52.2\,\delta_*$
\cite{CohnElkies,CT,CohnTable} --- a numerical estimate: outside
$d\in\{1,8,24\}$ the exact optimum of the Cohn--Elkies program is not
known, and computed values are approximations from above
\cite{CT}. Using our present dual point, strict non-sharpness would
need an independent upper bound below $32.91\,\delta_*$, i.e.\ an improvement of
$\approx37\%$ over the two-point bound. For comparison, the only general-purpose
technology quantified to improve on the two-point bound in this
range --- the three-point bounds of \cite{CdLS} --- gains $4.85\%$ at $d=12$ and
$0.79\%$ at $d=16$, and no three-point computation has been reported in
any dimension beyond $16$ (cf.\ the computational range of \cite{CdLS};
survey \cite{VallentinSurvey}). The paradox of
dimension $36$ is that the \emph{weakness} of the known record makes
the dual statement strong and the strict statement inaccessible: both
our dual value and the two-point bound are large multiples of the
record precisely because the record is weak; the residual gap between
them is plausibly due largely to our restricted search subspace and the
choice $T=10$, although the exact value of $\delta_{\LP}(36)$ is not
known.

\paragraph{Dimension 40.} The same pipeline reaches $d=40$
(weight $20$) in principle, with two new costs: the weight-$20$
eta-quotients of level $24$ span only a $72$-dimensional subspace of the
$80$-dimensional $M_{20}(\Gamma_0(24))$, so the basis must be augmented
(e.g.\ by $E_4\cdot(\text{weight-16 forms})$), and the exact arithmetic
is heavier ($\sigma_{19}$ growth). We have not completed a $d=40$
certificate. We note that Remark~4.3 of \cite{Zhou} asserts that for
$d=32$ and $d=40$ ``LP non-sharpness is established by other means'',
attributing both to Cohn--Triantafillou; but \cite{CT} treats only
$d\in\{12,16,20,28,32\}$ and does not address $d=40$ (and \cite{Zhou}'s
own Table~1 marks the status at $d=40$ as ``expected''). As of July 2026
we have been unable to locate any published dual bound or non-sharpness
proof in dimension $40$. The
modular-bootstrap analysis of \cite{AJCHdLT} (building on the sphere
packing--quantum gravity correspondence of \cite{HMR}) rules out
\emph{sharpness}
of the LP bound for a range of dimensions by an implied-kissing-number
criterion, but reports its first crossings of that criterion only near
$d\approx180$ and, in its authors' own words, ``does not give a
quantitative improvement''; it therefore neither contains nor implies a
dual bound at $d=36$ or $40$.

\paragraph{Methodological remarks.} Two ingredients may be reusable.
(i) Constraint generation appears to be the right formulation of the
exact dual LP whenever eventual positivity must be enforced through a
crossover: it keeps every exact solve small (here: one re-solve with
$73$ rows) while the dense exact LP is intractable and floating-point
LP is unsound at these conditionings. (ii) Lemma~\ref{lem:liftaware} is
elementary, but the failure of the unweighted constant in dimension
$36$ suggests it should be the default bookkeeping in
eventual-positivity verifications: deep oldform lifts are exactly the
coordinates the unweighted $\ell^1$ norm overweights, by the factor
$e^{(k-1)/2}$ that the lemma restores.

\paragraph{Computer-assisted proof and the trusted computing base.}
Table~\ref{tab:tcb} lists the rigorous pipeline behind
Theorem~\ref{thm:main}: each stage, the program that performs it, the
arithmetic model, and the certified output. Two points deserve emphasis.
First, floating point enters at exactly two places --- the construction
heuristics (\S\ref{sec:space}, pivot-basis selection) and the eigenspace
split (\S\ref{sec:tail}) --- and neither is trusted: each is validated
downstream by an exact or interval gate (respectively the exact
feasibility scan and the exact reconstruction identity). Second, the
driver \texttt{verify\_certificate.py} runs the exact-rational stages end
to end but takes the two $C_w$ upper bounds as \emph{separately
certified} inputs; their interval certification is the job of
\texttt{d36\_cs\_certificate.py} (re-verified in Arb by
\texttt{d36\_arb\_check.py}), so ``one command verifies the certificate''
holds only once these two scripts have both been run. Pinned library
versions and SHA-256 checksums of the data files are in
\texttt{requirements.txt} and \texttt{SHA256SUMS}.

\begin{table}[htbp]
\centering\small
\begin{tabular}{@{}llll@{}}
\toprule
Stage & Program & Arithmetic & Certified output / gate \\
\midrule
Reduce $g,\tilde g$ & \texttt{verify\_certificate} & exact $\mathbb{Q}$ & $a_0{=}1,\ a_{1\ldots9}{=}0$; ranks $29,10$ \\
Feasibility scan & \texttt{verify\_certificate} & exact $\mathbb{Q}$ & $a_n,b_n\ge0\ (1\le n\le800)$ \\
Eisenstein floor & \texttt{d36\_guarded\_cE} & exact $\mathbb{Q}$ & $c_E(\gamma)>0$ (16/16); $n_0{=}63,25$ \\
Newform decomp. & \texttt{d36\_cs\_certificate} & exact $\mathbb{Q}$ (split: float) & reconstruction dev.\ $=0$ to $800$ \\
Enclose $C_w$ & \texttt{d36\_cs\_certificate} & interval (\texttt{mpmath.iv}) & $C_w\le\eqref{eq:Cw}$; min pivot $>0$ \\
Re-certify $C_w$ & \texttt{d36\_arb\_check} & Arb ball arithmetic & same bound; system nonsingular \\
\bottomrule
\end{tabular}
\caption{The trusted computing base for Theorem~\ref{thm:main}. Program
names are abbreviated (drop the \texttt{.py}); the full file map, with
primary-versus-derived inputs, is \texttt{ancillary/MANIFEST.md}.}
\label{tab:tcb}
\end{table}

\paragraph{Reproducibility.} The exact verification --- the rigorous part
of the certificate --- uses only exact rational linear algebra, integer
$q$-expansions of eta-quotients, and interval arithmetic (Python:
\texttt{sympy}, \texttt{mpmath}, \texttt{fractions}); no computer-algebra
system with modular-forms machinery is required. Floating-point routines
(\texttt{numpy}, \texttt{scipy}) are used only in the \emph{construction}
step, as heuristics to select the pivot basis (\S\ref{sec:space}) and
to locate the candidate vertex; they play
no role in the exact rigor pass, and every theorem-critical constant
reported here is either recomputed exactly or bounded by a certified
outward enclosure (the labeled numerical estimates of this section are
the only exceptions). The full pipeline, receipts, and the exact
certificate are provided as ancillary files, together with an
independent second implementation of the newform decomposition
(high-precision simultaneous diagonalization) that reproduces the
constants \eqref{eq:Cw} to all reported digits; a third, optional
CAS-based cross-check --- a PARI/GP reproduction
(\texttt{d36\_C\_pari\_indep.py}, via the native modular-symbols
engine, no shared code path) that rebuilds the $64$-form
newform/oldform lift basis from scratch and again reproduces both
constants \eqref{eq:Cw} to all reported digits, along with the
newspace dimension table and the level-one eigenvalues; and an
independent Arb (ball-arithmetic) re-verification of the
theorem-critical interval solve (\texttt{d36\_arb\_check.py},
\S\ref{sec:tail}). The ancillary manifest
separates the files that constitute the \emph{proof} of
Theorem~\ref{thm:main} from exploratory/receipt material, lists SHA-256
checksums of the data files, and the driver
\texttt{verify\_certificate.py} re-runs the exact-data chain ---
reduced-basis reconstruction from the published form list, vertex
feasibility scans, Eisenstein projection, guarded $c_E$, tail
crossover --- from the published exact data alone (with the $C_w$
upper bounds as separately certified inputs), printing
\texttt{VERIFIED} only if every gate passes.

\paragraph{Data and code availability.} The exact certificate data, the
construction and verification receipts, and the complete Python pipeline
(including the independent second implementation) are provided as
ancillary files accompanying this submission; see
\texttt{ancillary/MANIFEST.md} for the file map and reproduction order.

\paragraph*{Acknowledgements.} The construction, the exact-arithmetic
certificate, and the preparation of this manuscript were carried out with
substantial assistance from AI language-model tools: Anthropic's Claude
models (Claude Fable~5 and Claude Opus~4.8) were used for the construction
of the certificate, the verification tooling, and the drafting of the
manuscript, and OpenAI's GPT-5.6 (Codex) was used as an independent
cross-check of the certificate computations. The
correctness of the main result does not rest on trust in those tools: the
certificate is exact and machine-checkable, and every numerical claim has
been reproduced by re-running the exact-arithmetic verification from
scratch (the pipeline and receipts are provided as ancillary files). The
author is accountable for this work and for the decision to publish it.

\subsection*{Mathematics Subject Classification (2020)}
Primary 52C17 (Packing and covering in $n$ dimensions);
Secondary 11F11 (Holomorphic modular forms of integral weight),
11F30 (Fourier coefficients of automorphic forms),
90C05 (Linear programming).

\medskip\noindent\emph{Keywords:} sphere packing, linear programming
bound, Cohn--Elkies bound, dual bound, modular forms, eta-quotients,
eventual positivity, exact arithmetic.

\end{document}